\documentclass[11pt,leqno,a4paper]{amsart}

\usepackage{hyperref}
\usepackage{latexsym}
\usepackage{amssymb,amsmath}
\usepackage{graphics}
\usepackage{url}
\usepackage{bbm}
\usepackage{multirow}
\usepackage[vcentermath]{youngtab}
  
\usepackage{booktabs}  
\usepackage{caption}
\usepackage{tikz}
\tikzstyle{empty}=[circle,draw=black!80,thick]
\tikzstyle{emptyn}=[circle,draw=black!80,fill=white,scale=0.5] 
\tikzstyle{nero}=[circle,draw=black!80,fill=black!80,thick]

\usepackage{amssymb,enumerate}
\newcommand{\Irr}{{\operatorname{Irr}}}
\newcommand{\fS}{\mathfrak{S}}
\newcommand{\fA}{\mathfrak{A}}

\newcommand{\triv}{\mathbbm{1}}
\newcommand{\innprod}[2]{\langle \chi^{#1}\big\downarrow_{#2}, \triv_{#2} \rangle}
\newcommand{\down}{\big\downarrow}
\newcommand{\up}{\big\uparrow}

\newtheorem{thm}{Theorem}[section]
\newtheorem{lem}[thm]{Lemma}

\newtheorem{prop}[thm]{Proposition}

\newtheorem*{thmA}{Theorem A}
\newtheorem*{thmB}{Theorem C}
\newtheorem*{corC}{Corollary B}
\newtheorem*{corD}{Corollary D}
\newtheorem*{Question}{Question}

\theoremstyle{definition}
\newtheorem{exmp}[thm]{Example}
\newtheorem{defn}[thm]{Definition}

\newtheorem{rem}[thm]{Remark}

\begin{document}

\title{On permutation characters and Sylow $p$-subgroups of $\fS_n$}


\author{Eugenio Giannelli}
\address[E. Giannelli]{Department of Pure Mathematics and Mathematical Statistics, University of Cambridge, Cambridge CB3 0WA, UK}
\email{eg513@cam.ac.uk}

\author{Stacey Law}
\address[S. Law]{Department of Pure Mathematics and Mathematical Statistics, University of Cambridge, Cambridge CB3 0WA, UK}
\email{swcl2@cam.ac.uk}

\thanks{The first author's research was funded by Trinity Hall, University of Cambridge.}
 
\begin{abstract}
Let $p$ be an odd prime and let $n$ be a natural number. In this article we determine the irreducible constituents of the permutation module induced by the action of the symmetric group $\fS_n$ on the cosets of a Sylow $p$-subgroup $P_n$.
As a consequence, we determine the number of irreducible representations of the corresponding Hecke algebra $\mathcal{H}(\fS_n, P_n, \triv_{P_n})$.
\end{abstract}



\maketitle

\section{Introduction}\label{sec:intro}
In this article we answer a question of Alex Zalesski (private communication with the first author) concerning the decomposition into irreducible constituents of the permutation character $(\triv_{P_n})\up^{\fS_n}$, where $\fS_n$ is the symmetric group of degree $n$, $p$ is an odd prime and $P_n$ is a Sylow $p$-subgroup of $\fS_n$.
More precisely, our main result determines all of the irreducible constituents of $(\triv_{P_n})\up^{\fS_n}$ in characteristic 0.
We recall that the set $\Irr(\fS_n)$ of ordinary irreducible characters of the symmetric group is naturally in bijection with the set $\mathcal{P}(n)$ of partitions of $n$. 
For any $\lambda\in\mathcal{P}(n)$ we let $\chi^\lambda\in\mathrm{Irr}(\fS_n)$ be the corresponding irreducible character. 

\begin{thmA} 
Let $p\geq 5$ be a prime, let $n$ be a natural number and let $\lambda\in\mathcal{P}(n)$. Then $\chi^\lambda$ is not an irreducible constituent of $(\triv_{P_n})\up^{\fS_n}$ if and only if $n=p^k$ for some $k\in\mathbb{N}$ and $\lambda\in\{(p^k-1,1), (2,1^{p^k-2}) \}$.

If $p=3$, then $\chi^\lambda$ is not an irreducible constituent of $(\triv_{P_n})\up^{\fS_n}$ if and only if $n=3^k$ for some $k\in\mathbb{N}$ and $\lambda\in\{(3^k-1,1), (2,1^{3^k-2}) \}$, or $n\leq 10$ and $\lambda$ is one of the following partitions: $$(2,2);\ \ (3,2,1);\ \ (5,4), (2^4,1), (4,3,2), (3^2,2,1);\ \ (5,5),(2^5).$$
\end{thmA}

Ignoring for a moment the few exceptions arising for small symmetric groups at the prime $3$, Theorem A shows that given any natural number $n\in\mathbb{N}$ which is not a power of $p$, the restriction to $P_n$ of any irreducible character of $\fS_n$ has the trivial character $\triv_{P_n}$ as a constituent.
We remark immediately that this clearly does not hold for $p=2$. For instance, the sign representation of $\fS_n$ restricts irreducibly and non-trivially to a Sylow $2$-subgroup of $\fS_n$. More generally, when $n$ is a power of $2$, \cite[Theorem 1.1]{G} shows that no non-trivial irreducible character of odd degree of $\fS_{n}$ appears as an irreducible constituent of $(\triv_{P_n})\up^{\fS_n}$, where $P_n\in\mathrm{Syl}_2(\fS_n)$.
The above observations underline that for the prime $2$ the situation is notably less regular than for odd primes; we believe that in this case a very large proportion (almost half) of the irreducible characters of $\fS_n$ are not irreducible constituents of the discussed permutation character. Nevertheless, at the time of this writing we do not have a conjecture for a characterization of the subset of partitions of $n$ labelling irreducible characters appearing as constituents of $(\triv_{P_n})\up^{\fS_n}$ when $p=2$.

\medskip

Let $\mathcal{H}:=\mathcal{H}(\fS_n,P_n,\triv_{P_n})$ be the Hecke algebra naturally corresponding to the permutation character $(\triv_{P_n})\up^{\fS_n}$. (We refer the reader to \cite[Chapter 11D]{CR} for the complete definition and properties of this correspondence.)
It is well known that the number of irreducible representations of $\mathcal{H}$ equals the number of distinct irreducible constituents of the corresponding permutation character (see for example \cite[Theorem (11.25)(ii)]{CR}). 
Therefore our Theorem A has the following consequence. 

\begin{corC}
Let $p$ be an odd prime and let $n>10$ be a natural number. 
If $n\neq p^k$ (respectively $n=p^k$) then the Hecke algebra $\mathcal{H}$ has exactly $|\mathcal{P}(n)|$ (respectively $|\mathcal{P}(n)|-2$) 
irreducible representations. 
\end{corC}

As explained in \cite[Theorem 11.25(iii)]{CR}, understanding the dimensions of the irreducible representations of $\mathcal{H}$ is equivalent to determining the multiplicities of the irreducible constituents of $(\triv_{P_n})\up^{\fS_n}$.
For this reason we believe that it would be extremely interesting to find a solution to the following problem. 

\begin{Question}
Is there a combinatorial description of the map $f: \mathcal{P}(n)\longrightarrow \mathbb{N}_0$, where $f(\lambda)$ equals the multiplicity of $\chi^\lambda$ as an irreducible constituent of $(\triv_{P_n})\up^{\fS_n}$ ?
\end{Question}

\medskip

A second consequence of Theorem A is a precise description of the constituents of the permutation character $(\triv_{Q_n})\up^{\fA_n}$, where $\fA_n$ is the alternating group of degree $n$ and $Q_n$ is a Sylow $p$-subgroup of $\fA_n$. We recall that the ordinary irreducible characters of the alternating group $\mathfrak{A}_n$ can be labelled as $\Irr(\mathfrak{A}_n) = \{\chi^\lambda\down_{\mathfrak{A}_n}\ |\ \lambda\ne\lambda'\in\mathcal{P}(n) \}\ \cup\ \{ \chi^{\lambda\pm}\ |\ \lambda=\lambda'\in\mathcal{P}(n) \}$, where $\lambda'$ is the partition conjugate to $\lambda$ (see \cite[Chapter 2.5]{JK}). 

\begin{thmB}
Let $p\ge 5$ be a prime, let $n$ be a natural number and let $\psi\in\Irr(\mathfrak{A}_n)$. Then $\psi$ is not an irreducible constituent of $(\triv_{Q_n})\up^{\mathfrak{A}_n}$ if and only if $n=p^k$ for some $k\in\mathbb{N}$ and $\psi=\chi^\lambda\down_{\mathfrak{A}_n}$ with $\lambda\in\{(p^k-1,1), (2,1^{p^k-2}) \}$.

If $p=3$, then $\psi\in\Irr(\mathfrak{A}_n)$ is not an irreducible constituent of $(\triv_{Q_n})\up^{\mathfrak{A}_n}$ if and only if $n=3^k$ for some $k\ge 2$ and $\psi=\chi^\lambda\down_{\mathfrak{A}_n}$ with $\lambda\in\{(3^k-1,1), (2,1^{3^k-2}) \}$, or $n\le 10$ and $\psi\in\{\chi^{(2,1)\pm}, \chi^{(2,2)\pm}, \chi^{(3,2,1)\pm}, \chi^\lambda\down_{\mathfrak{A}_n} \}$ where $\lambda\in\{(5,4),(2^4,1),(4,3,2),(3^2,2,1),(5^2),(2^5) \}$.
\end{thmB}

Theorem C follows immediately from Theorem A by observing that when $p$ is odd, $Q_n$ is a Sylow $p$-subgroup of $\fS_n$.

\medskip

We conclude by mentioning that Theorem A gives information on the eigenvalues of the irreducible representations of $\fS_n$, at elements of odd prime power order.
This may already be known to experts, but we were not able to find a reference in the literature.

\begin{corD}
Let $p\ge 5$ be a prime and let $n$ be a natural number. Let $\lambda\in\mathcal{P}(n)$ and let $\Xi^\lambda$ be the representation of $\mathfrak{S}_n$ affording $\chi^\lambda$. If $n$ is not a power of $p$, or if $n=p^k$ but $\lambda\notin\{(p^k-1,1),(2,1^{p^k-2})\}$, then $\Xi^\lambda(g)$ has an eigenvalue equal to $1$ for any $g\in\mathfrak{S}_n$ of prime power order.
\end{corD}

We remark that an analogous study was done extensively in \cite{ZZ} in the case of Chevalley groups.
The case of elements of prime order was discussed in \cite{Z} for quasi-simple groups.


\section{Notation and Preliminaries}\label{sec:2}
\subsection{Characters of wreath products}
Let $G$ be a finite group and let $R$ be a subgroup of $\fS_n$ for some $n\in\mathbb{N}$. We denote by $G^{\times n}$ the external direct product of $n$ copies of $G$. The natural action of $\fS_n$ on the direct factors of $G^{\times n}$ induces an action via automorphisms of $\fS_n$ (and therefore of $R\le \fS_n$) on $G^{\times n}$, giving the wreath product $G\wr R:= G^{\times n}\rtimes R$. As in \cite[Chapter 4]{JK}, we denote the elements of $G\wr R$ by $(g_1,\dotsc,g_n;r)$ for $g_i\in G$ and $r\in R$. Let $V$ be a $\mathbb{C}G$-module affording the character $\phi$, with $\mathbb{C}$-basis $e_1,\dotsc,e_d$. We let $V^{\otimes n}:=\underbrace{V\otimes\cdots\otimes V}_n$ be the corresponding $\mathbb{C}G^{\times n}$-module. The left action of $G\wr R$ on $V^{\otimes n}$ defined by
$$e_{i_1}\otimes\cdots\otimes e_{i_n}\mapsto g_1(e_{i_{r^{-1}(1)}})\otimes\cdots\otimes g_n(e_{i_{r^{-1}(n)}}),\qquad i_1,\dotsc,i_n\in\{1,\dotsc,d\}$$
for any element $(g_1,\dotsc,g_n;r)\in G\wr R$ turns $V^{\otimes n}$ into a $\mathbb{C}(G\wr R)$-module. We denote by $\hat{\phi}$ the character afforded by $V^{\otimes n}$ as a $\mathbb{C}(G\wr R)$-module. For any ordinary character $\psi$ of $R$, we let $\psi$ also denote its inflation to $G\wr R$ and let
$$\mathcal{X}(\phi;\psi):=\hat{\phi}\cdot\psi$$
be the ordinary character of $G\wr R$ obtained as the inner tensor product of $\hat{\phi}$ and $\psi$.

Let $\phi\in\Irr(G)$ and let $\phi^{\times n}:=\phi\times\cdots\times\phi$ denote the corresponding irreducible character of $G^{\times n}$. From the description of irreducible characters of wreath products, for example in \cite[Chapter 4]{JK}, we deduce that $\hat{\phi}\in\Irr(G\wr R)$ is an extension of $\phi^{\times n}$. For $\psi\in\Irr(R)$ we have that
$$\mathcal{X}(\phi;\psi)=\hat{\phi}\cdot\psi\in\Irr(G\wr R\ |\ \phi^{\times n}),$$
where $\Irr(G\wr R\ |\ \phi^{\times n})$ is the set of irreducible characters $\chi$ of $G\wr R$ whose restriction $\chi\down_{G^{\times n}}$ contains $\phi^{\times n}$ as an irreducible constituent. Indeed, Gallagher's Theorem \cite[Corollary 6.17]{I} gives
$$\Irr(G\wr R\ |\ \phi^{\times n}) = \{\mathcal{X}(\phi;\psi)\ |\ \psi\in\Irr(R)\}.$$
More generally, if $H\le G$ and $\psi\in\Irr(H)$ then we denote by $\Irr(G\ |\ \psi)$ the set of characters $\chi\in\Irr(G)$ such that $\psi$ is an irreducible constituent of the restriction $\chi\down_H$.

\subsection{The representation theory of $\fS_n$ and their Sylow $p$-subgroups}
For each $n\in\mathbb{N}$, $\mathrm{Irr}(\fS_n)$ is naturally in bijection with $\mathcal{P}(n)$, the set of all partitions of $n$. For $\lambda\in\mathcal{P}(n)$ (also written $\lambda\vdash n$), the corresponding irreducible character is denoted by $\chi^\lambda$. We recall that the Young diagram $[\lambda]$ corresponding to the partition $\lambda=(\lambda_1,\lambda_2,\dotsc,\lambda_k)$ is the subset of the plane defined by: 
$$[\lambda]=\{(i,j)\in{\mathbb N}\times{\mathbb N}\mid 1\leq i\leq k,\ 1\leq j\leq\lambda_i\},$$
where we view the diagram in matrix orientation, with the node $(1,1)$ in the upper left corner. For $\lambda$ a partition, $\lambda'$ denotes the partition conjugate to $\lambda$. The size of $\lambda$ is denoted by $|\lambda|$; that is, $\lambda\vdash|\lambda|$.

\smallskip

Let $n\in\mathbb{N}$. Throughout this article, $p$ denotes a fixed odd prime and $P_n$ denotes a Sylow $p$-subgroup of $\fS_n$. We recall some facts about Sylow subgroups of symmetric groups, and refer the reader to \cite[Chapter 4]{JK} for a more detailed discussion. Clearly $P_1$ is the trivial group while $P_p$ is cyclic of order $p$. More generally, $P_{p^i}= (P_{p^{i-1}})^{\times p}\rtimes P_p=P_{p^{i-1}}\wr P_p\cong P_p\wr \cdots \wr P_p$ ($p$-fold wreath product).
Let $n=\sum_{i=0}^t b_ip^i$ be the $p$-adic expansion of $n$. Then $P_n\cong P_{p^0}^{\times b_0}\times P_{p^1}^{\times b_1}\times\cdots\times P_{p^t}^{\times b_t}.$

\subsection{The Littlewood\textendash Richardson Rule}
Let $m,n\in\mathbb{N}$ with $m<n$. For $\mu\vdash m$ and $\nu\vdash n-m$, the Littlewood\textendash Richardson rule (see \cite[Chapter 16]{J}) describes the decomposition into irreducible constituents of induced character
$$(\chi^\mu\times\chi^\nu)\up^{\fS_n}_{\fS_m\times \fS_{n-m}}.$$

Before we recall the Littlewood\textendash Richardson rule, we introduce some notation and technical definitions. By a skew shape $\gamma$ we mean a set difference of Young diagrams $[\lambda\setminus\mu]$ for some partitions $\lambda$ and $\mu$ with $[\mu]\subsetneq[\lambda]$, and $|\gamma|:=|\lambda|-|\mu|$. By convention, the highest row of $[\lambda]$ for a partition $\lambda$ is numbered 1, but the highest row of a skew shape $\gamma=[\lambda\setminus\mu]$ need not be the highest row of $[\lambda]$.

\begin{defn}
Let $\lambda=(\lambda_1,\dotsc,\lambda_k)\in\mathcal{P}(n)$ and let $\mathcal{C}=(c_1,\dotsc,c_n)$ be a sequence of positive integers. We say that $\mathcal{C}$ is of \emph{type $\lambda$} if
$$|\{i\in\{1,\dotsc,n\}\ :\ c_i=j\}| = \lambda_j$$
for all $j\in\{1,\dotsc,k\}$. We say that an element $c_j$ of $\mathcal{C}$ is \emph{good} if $c_j=1$ or if
$$|\{i\in\{1,2,\dotsc,j-1\}\ :\ c_i=c_j-1\}|>|\{i\in\{1,2,\dots,j-1\}\ :\ c_i=c_j\}|.$$
Finally, we say that the sequence $\mathcal{C}$ is \emph{good} if $c_j$ is good for every $j\in\{1,\dotsc,n\}$.
\end{defn}

\begin{thm}\label{thm:LR}[Littlewood\textendash Richardson rule]
Let $m,n\in\mathbb{N}$ with $m<n$. Let $\mu\vdash m$ and $\nu\vdash n-m$. Then
$$(\chi^\mu\times\chi^\nu)\up^{\fS_n}_{\fS_m\times \fS_{n-m}} = \sum_{\lambda\vdash n} c^\lambda_{\mu\nu}\ \chi^\lambda$$
where $c^\lambda_{\mu\nu}$ equals the number of ways to replace the nodes of $[\lambda\setminus\mu]$ by natural numbers such that
\begin{itemize}
\item[(i)] The sequence obtained by reading the numbers from right to left, top to bottom is a good sequence of type $\nu$;
\item[(ii)] The numbers are weakly increasing along rows;
\item[(iii)] The numbers are strictly increasing down columns.
\end{itemize}
\end{thm}

Let $\nu$ be a partition. We call a way of replacing the nodes of a skew shape $\gamma$ with $|\nu|$ boxes by numbers satisfying conditions (i)-(iii) of Theorem~\ref{thm:LR} a \emph{Littlewood\textendash Richardson filling of $\gamma$ of type $\nu$}. It is easy to see that every skew shape has at least one Littlewood\textendash Richardson filling. Moreover, the coefficients described in Theorem~\ref{thm:LR} are symmetric: $c^\lambda_{\mu\nu}=c^\lambda_{\nu\mu}$ for all partitions $\mu$, $\nu$ and all partitions $\lambda\vdash |\mu|+|\nu|$.
For convenience, let $\mathcal{LR}(\gamma)$ denote the set of all possible types of Littlewood\textendash Richardson fillings of a skew shape $\gamma$. We write $\gamma\cong[\lambda]$ if $\gamma$ is a translation of the Young diagram of the partition $\lambda$, and we denote by $\gamma^\circ$ the $180^\circ$-rotation of $\gamma$ (up to translation).

We record below three useful lemmas.

\begin{lem}\label{lem:BK}\cite[Lemma 4.4]{BK}
Let $\mu$ and $\gamma$ be partitions such that $[\gamma]\subsetneq[\mu]$. The following are equivalent:
\begin{itemize}
\item[(i)] $|\mathcal{LR}([\mu\setminus\gamma])|=1$;
\item[(ii)] There is a unique Littlewood\textendash Richardson filling of $[\mu\setminus\gamma]$;
\item[(iii)] $[\mu\setminus\gamma]\cong[\nu]$ or $[\mu\setminus\gamma]^\circ\cong[\nu]$, for some partition $\nu\vdash|\mu|-|\gamma|$.
\end{itemize}
\end{lem}

\begin{lem}\label{lem:filling}
Let $\gamma$ be a skew shape. Suppose the non-empty rows of $\gamma$ are numbered $1\le r_1<r_2<\dotsc<r_t$. Then in any Littlewood\textendash Richardson filling of $\gamma$, the boxes in row $r_i$ can only be filled with numbers from $\{1,2,\dotsc,i\}$, for all $1\le i\le t$.
\end{lem}

\begin{proof}
This is immediate from conditions (i)-(iii) of Theorem~\ref{thm:LR}.
\end{proof}

\begin{lem}\label{lem:filling-type-(m-1,1)}
Let $\gamma$ be a skew shape and let $m=|\gamma|\ge 4$. Suppose $(m-1,1)\in\mathcal{LR}(\gamma)$. Then one of the following holds:

(i) $\gamma\cong[(m-1,1)]$ or $\gamma^\circ\cong[(m-1,1)]$;

(ii) $\mathcal{LR}(\gamma)\cap \{ (m), (m-2,2), (m-2,1,1)\} \ne \emptyset$.
\end{lem}

\begin{proof}
Since $(m-1,1)\in\mathcal{LR}(\gamma)$, no three boxes of $\gamma$ lie in the same column, and $\gamma$ has at most one column containing two boxes. Suppose (i) does not hold. Then $(m)\in\mathcal{LR}(\gamma)$ if (a) no two boxes of $\gamma$ lie in the same column; or $(m-2,1,1)\in\mathcal{LR}(\gamma)$ if (b) $\gamma$ has precisely two connected components, one of which is a row of $m-2$ boxes and the other of which is a column of two boxes. 

Now assume $\gamma$ satisfies neither (a) nor (b). Then $\gamma$ has a unique connected component $\delta$ whose boxes lie in exactly two rows, say rows $j$ and $j+1$, and each of the other components lies entirely within one row. Moreover, if $\delta=\gamma$ is the unique connected component then $\delta$ has at least two boxes in each of rows $j$ and $j+1$, while if $\delta$ contains only two boxes then by assumption $\gamma$ has at least three connected components. In all instances, $(m-2,2)\in\mathcal{LR}(\gamma)$.
\end{proof}


\section{The prime power case}\label{sec:outline}

Let $p$ be an odd prime. 
The aim of this section is to prove Theorem A for $n=p^k$. As we will see, this is the crucial part of the article. In fact, the complete statement for all natural numbers follows relatively easily from the prime power case. 
\begin{defn}
From now on we let $\Delta(p^k)=\mathcal{P}(p^k)\smallsetminus \{(p^k-1,1), (2,1^{p^k-2}) \}.$
\end{defn}
For the convenience of the reader, we state here the main object of our section. 
\begin{thm}\label{thm:n-prime-power}
Let $k\in\mathbb{N}$ and $\lambda\vdash p^k\ne 9$. Then $\innprod{\lambda}{P_{p^k}} = 0$ if and only if $\lambda\notin\Delta(p^k)$. If $p^k=9$ then $\innprod{\lambda}{P_9} = 0$ if and only if $\lambda\in\{(8,1),(5,4),(4,3,2),(3^2,2,1),(2^4,1),(2,1^7)\}.$
\end{thm}

Our proof is by induction on $k\in\mathbb{N}$. We start by stating and proving the base case $k=1$.

\begin{lem}\label{lem:n<p}
Let $n\in\mathbb{N}$ and suppose $n\le p$. Let $\lambda\vdash n$. Then $\innprod{\lambda}{P_n} = 0$ if and only if $n=p$ and $\lambda\in\{(p-1,1),(2,1^{p-2})\}$.
\end{lem}

\begin{proof}
If $n<p$ then $P_n=1$ and the statement obviously holds. If $n=p$ then $P_p$ is cyclic of order $p$ and so
$$ \innprod{\lambda}{P_n} = \tfrac{1}{p}\big( (p-1)\cdot\chi^\lambda(\sigma) + \chi^\lambda(1) \big),$$
where $\sigma$ is a $p$-cycle in $\fS_p$. By the Murnaghan\textendash Nakayama rule (see \cite[2.4.7]{JK}), $\chi^\lambda(\sigma) = (-1)^l$ if $\lambda$ is a hook and $l$ is the leg length of $\lambda$, and $\chi^\lambda(\sigma)=0$ otherwise. Thus $\innprod{\lambda}{P_n} = 0$ if and only if $\lambda$ is a hook of odd leg length and $\chi^\lambda(1)=p-1$. This holds if and only if $\lambda\in\{(p-1,1),(2,1^{p-2})\}$.
\end{proof}

We record another easy and useful fact.
\begin{lem}\label{lem:conj}
Let $n\in\mathbb{N}$ and $\lambda\vdash n$. Then $\chi^\lambda\down_{P_n} = \chi^{\lambda'}\down_{P_n}$.
\end{lem}

\begin{proof}
We know that $\chi^{\lambda'}=\chi^\lambda\cdot\operatorname{sign}(n)$ where $\operatorname{sign}(n)$ is the sign character of $\fS_n$. Since $p$ is odd, $P_n$ is contained in the alternating subgroup of $\fS_n$, and the assertion follows.
\end{proof}

The following proposition is one of the key steps in our proof of Theorem \ref{thm:n-prime-power}.

\begin{prop}\label{prop:if-property-star}
Let $\{\mu_1,\ldots, \mu_p\}$ be a subset of $\mathcal{P}(p^k)$ of size at least $2$, such that for all $1\le i\le p$, we have $\innprod{\mu_i}{P_{p^k}}\neq 0$. Let $\lambda\in\mathcal{P}(p^{k+1})$ be such that $\chi^{\mu_1}\times\cdots\times\chi^{\mu_p}$ is an irreducible constituent of $(\chi^\lambda)\down_{\fS_{p^k}^{\times p}}$. Then $\innprod{\lambda}{P_{p^{k+1}}} \neq 0$.
\end{prop}
\begin{proof}
Let $G=\fS_{p^{k+1}}$, let $B=\fS_{p^k}^{\times p}\leq \fS_{p^{k+1}}$ and set $\psi=\chi^{\mu_1}\times\cdots\times\chi^{\mu_p}\in\mathrm{Irr}(B)$. 
Let $P=P_{p^{k+1}}$ be such that $P=C\rtimes D$ where $P_{p^k}^{\times p}\cong C\leq B$ and where $P_p\cong D\leq G$.
Let $R$ be a subgroup of $G$ isomorphic to $\fS_p$, naturally acting on $B$ by permuting (as blocks for its action) the $p$ direct factors of $B$. Hence $H:=B\rtimes R$ is such that 
$B\leq H\leq G$ and $H\cong \fS_{p^k}\wr\fS_p$. We choose $R$ containing $D$ in order to have $P\leq H$.
Since $\chi^\lambda\in\mathrm{Irr}(G\ |\ \psi)$, there exists $\chi\in\mathrm{Irr}(H\ |\ \psi)$ such that $\langle \chi^{\lambda}\down_{H}, \chi \rangle\neq 0$.

Let $A=\{\mu_1,\dotsc,\mu_p\}$ and let $a=|A|$. Then (without loss of generality, up to a reordering of the partitions in $A$) there exist partitions $\gamma_1, \ldots, \gamma_a\in\mathcal{P}(p^k)$ and a partition $(i_1,i_2,\ldots, i_a)\vdash p$ such that 
$\gamma_1 = \mu_1 = \mu_2 = \dotsc = \mu_{i_1}$, $\gamma_2 = \mu_{i_1+1} = \dotsc = \mu_{i_1+i_2}$, and so on.
Let $I$ be the subgroup of $R$ such that $B\rtimes I$ is the stabilizer of $\psi$ in $H$. Clearly $I\cong\fS_{i_1}\times\fS_{i_2}\times\cdots\times\fS_{i_{a}}\leq \fS_p$.

From the description of irreducible characters of wreath products given for example in \cite[Chapter~4.3]{JK}, we have that for all $j\in\{1,\ldots, a\}$ there exists $\nu_j\vdash i_j$ such that $\chi=\phi\up_{B\rtimes I}^H$, where 
$$ \phi = \mathcal{X}(\gamma_1;\nu_1)\times\cdots\times\mathcal{X}(\gamma_{a};\nu_{a})\in\mathrm{Irr}(B\rtimes I).$$
(Here we denoted by $\mathcal{X}(\gamma;\nu)$ the character $\mathcal{X}(\chi^\gamma;\chi^\nu)$. This lighter notation will be used again later on in the article.)
Recalling that $P= C\rtimes D$, Mackey's restriction formula gives
\begin{align*}
\chi\down_{P} &= \phi \up_{B\rtimes I}^{H} \down_{P}\\
&= \sum_{g\in P\setminus H/B\rtimes I} \Big( \phi^g\down_{(B\rtimes I^g)\cap (C\rtimes D)}^{(B\rtimes I^g)} \Big) \up ^{P}\\
&= \sum_{g\in P\setminus H/B\rtimes I} \Big( \phi^g\down_{C}^{(B\rtimes I^g)} \Big) \up ^{P}.
\end{align*}
The last equality holds because for all $g\in H$ we have that $I^g\cap D = 1$, since $D\cong P_p$ but $I$ contains no elements of order $p$ as $a\ge 2$.
Considering the double coset representative $g=1\in H$, we have that $\phi\down_C\up ^{P}$ is a summand of $\chi\down_{P}$. Moreover,
\begin{align*}
\phi\down_C
= N\cdot (\chi^{\gamma_1}\down_{P_{p^k}})^{\times i_1}\times\cdots\times (\chi^{\gamma_{a}}\down_{P_{p^k}})^{\times i_{a}}=N\big(\chi^{\mu_1}\down_{P_{p^k}}\times\cdots\times\chi^{\mu_p}\down_{P_{p^k}}\big),
\end{align*}
where $N=\prod_{j=1}^a \chi^{\nu_j}(1)\in\mathbb{N}$.
Hence, $\triv_C$ is an irreducible constituent of $\phi\down_C$ and therefore
$\triv_C\up^{P}$ is a summand of $\phi\down_C\up ^{P} $ and of $\chi\down_{P}$. This along with Frobenius reciprocity shows that 
$$\innprod{\lambda}{P} \ge \innprod{}{P} \ge \langle \triv_C\up^{P}, \triv_{P}\rangle = \langle \triv_C, \triv_P\down_C \rangle = 1.$$
\end{proof}

In light of Proposition \ref{prop:if-property-star}, we will now focus on the study of the restriction of irreducible characters of $\fS_{p^{k+1}}$ to the Young subgroup $\fS_{p^k}^{\times p}$.
For this reason, we introduce the following notation. 

\begin{defn}\label{def:star}
Let $k\in\mathbb{N}$ and let $q\in\{2,3,\dotsc,p\}$. 
We let $\mathcal{D}(q,p^k)$ be the subset of $\mathcal{P}(qp^k)$ consisting of all partitions $\lambda\vdash qp^k$ such that 
the restriction $(\chi^\lambda)\down_{\fS_{p^k}^{\times q}}$ has an irreducible constituent of the form $\chi^{\mu_1}\times\cdots\times\chi^{\mu_q}$ satisfying $\mu_i\in\Delta(p^k)$ for all $i$, and the $\mu_i$ are not all equal.
\end{defn}

Our next goal is to show that $\mathcal{D}(p,p^k)$ is a very large subset of $\Delta(p^{k+1})$. First we observe the following easy property. 

\begin{lem}\label{lem:star-conj}
Let $\lambda\in\mathcal{P}(qp^k)$. Then $\lambda\in\mathcal{D}(q,p^k)$ if and only if $\lambda'\in\mathcal{D}(q,p^k)$
\end{lem}

\begin{proof}
We know that $\chi^{\lambda'}=\chi^\lambda\cdot\mathrm{sign}(\fS_{qp^k})$. 
Moreover, we observe that the set $\Delta(p^k)$ is closed under conjugation of partitions. Since $(\mathrm{sign}(\fS_{qp^k}))\down_{\fS_{p^k}^{\times q}}=\mathrm{sign}(\fS_{p^k})\times\cdots\times\mathrm{sign}(\fS_{p^k})$, the statement follows.
\end{proof}

\begin{prop}\label{prop:eugenew}
Let $k\in\mathbb{N}$ be such that $p^k\neq 3$. Then 
$$\mathcal{D}(p,p^k)=\mathcal{P}(p^{k+1})\setminus \{(p^{k+1}),(p^{k+1}-1,1),(2,1^{p^{k+1}-2}),(1^{p^{k+1}})\}.$$
\end{prop}

\begin{proof}
The proof of this statement is postponed to Section \ref{sec:have-prop-star}.
\end{proof}

We are now ready to prove Theorem \ref{thm:n-prime-power}.

\medskip

\begin{proof}[Proof of Theorem~\ref{thm:n-prime-power}] We proceed by induction on $k\ge 1$ for $p\ge 5$ and on $k\ge 3$ for $p=3$. The base case for $p\ge 5$ follows from Lemma~\ref{lem:n<p}, while the assertion may be verified computationally for $k\le 3$ if $p=3$.
Now assume the statement holds for $k\in\mathbb{N}$. To ease the notation let $n=p^{k+1}$, $P=P_{n}$ and let $A$ be the set defined by $$A=\{\lambda\vdash n\ |\ \innprod{\lambda}{P} \neq 0 \}.$$
From Proposition \ref{prop:if-property-star}, used together with the inductive hypothesis, we deduce that $\mathcal{D}(p,p^k)\subseteq A$. Moreover, both $(n), (1^n)\in A$ since $\chi^{(n)}\down_{P} = \triv_{P} = \chi^{(1^{n})}\down_{P}$.
Hence we have that $\Delta(n)\subseteq A$, by Proposition \ref{prop:eugenew}.
To conclude we need to show that $(n-1,1)$ and $(2,1^{n-2})$ are not in $A$. 
By Lemma \ref{lem:conj} it suffices to show that $(n-1,1)\notin A$. 

Let $B=\fS_{p^k}^{\times p}\le \fS_{p^k}\wr\fS_p\le \fS_{p^k+1}$ and let $C\le B$ where $C\cong P_{p^k}^{\times p}$. From \cite[Lemma 3.2]{GTT} and the Littlewood\textendash Richardson rule we see that $\chi^{(n-1,1)}\down_B = (p-1)\triv_B+\Theta$, where 
$$\Theta=\sum_{i=0}^{p-1}(\chi^\mu\times\triv\times\cdots\times\triv)^{\sigma^i},\ \mu=(p^k-1,1)\ \text{and}\ \sigma\ \text{is a}\ p\text{-cycle in}\ (\fS_{p^k}\wr\fS_p)\smallsetminus B.$$
From \cite[Theorem 4.2 and Proposition 4.3]{G2} there exists $\nu\in\{(p-1,1), (2,1^{p-2})\}$ such that 
$$\chi^{(n-1,1)}\down_{\fS_{p^k}\wr\fS_p} = \mathcal{X}((p^k);\nu) + \Delta,$$ where $\Delta$ is a sum of irreducible characters of $\fS_{p^k}\wr\fS_p$ whose degree is divisible by $p$.
Since $\mathcal{X}((p^k);\nu)\down_B = (p-1)\triv_B$, we have that $\Delta\down_B=\Theta$. Using the inductive hypothesis, we see that $\triv_{P_{p^k}}$ is not a constituent of $\chi^\mu\down_{P_{p^k}}$. Hence $\langle\Theta\down_C,\triv_C\rangle=0$. Together these show that $\langle\Delta\down_C,\triv_C\rangle=\langle\Theta\down_C,\triv_C\rangle=0$ and we deduce that $\langle(\chi^{(n-1,1)})\down_P, \triv_P\rangle=\langle \mathcal{X}((p^k);\nu)\down_{P}, \triv_{P}\rangle$.

Finally by Lemma \ref{lem:n<p} we know that $\langle\chi^\nu\down_{P_p}, \triv_{P_p}\rangle=0$. Since $\triv_P=\mathcal{X}(\triv_{P_{p^k}}, \triv_{P_p})$ we deduce that 
$\langle \mathcal{X}((p^k);\nu)\down_{P}, \triv_{P}\rangle=0$, whence $\langle(\chi^{(n-1,1)})\down_P, \triv_P\rangle=0$ as required. Thus the statement of the theorem holds for $k+1$. This concludes the proof.
\end{proof}



\section{The proof of Proposition~\ref{prop:eugenew}}\label{sec:have-prop-star}

The goal of this section is to give a complete proof of Proposition \ref{prop:eugenew}. In order to do this we will show that a more general fact (see Proposition \ref{prop:have-property-star} below) holds. In order to state Proposition \ref{prop:have-property-star}, we first need to introduce the following notation. 

\begin{defn}
For $q\in\{3,4,\dotsc,p\}$, let
$$A(q,p^k):= \mathcal{P}(qp^k)\setminus \{(qp^k),(qp^k-1,1),(2,1^{qp^k-2}),(1^{qp^k})\}.$$
For $q=2$, let
$$A(2,p^k) := \mathcal{P}(2p^k)\setminus \{(2p^k),(2p^k-1,1),(p^k,p^k), (2^{p^k}),(2,1^{2p^k-2}),(1^{2p^k})\}.$$
\end{defn}

\smallskip

\begin{prop}\label{prop:have-property-star}
Let $k\in\mathbb{N}$ be such that $p^k\notin \{3,5\}$. Then $\mathcal{D}(q,p^k)=A(q,p^k)$, for all $q\in\{2,3,\dotsc,p\}$. 
\end{prop}

\begin{rem}\label{rem: p^k=5}
We excluded the case $p^k=5$ from the statement of Proposition \ref{prop:have-property-star} for the sole reason that this statement is false if $q=2$ and $p^k=5$. 
Nevertheless Proposition \ref{prop:eugenew} holds for $p^k=5$, as we will remark at the end of this section. 
\end{rem}

The rest of this section is devoted to proving Proposition \ref{prop:have-property-star}, by induction on $q$.
A fundamental tool is 
the operator $\Omega_q:\mathcal{P}(qn)\rightarrow\mathcal{P}((q-1)n)$.
This was first defined in \cite[Section 3]{GN} and it is recalled below for the convenience of the reader. Given compositions $\mu$ and $\nu$, we denote by $\mu^\star$ the unique partition obtained by reordering the parts of $\mu$, and by $\mu\circ\nu$ the composition of $|\mu|+|\nu|$ obtained by concatenating $\mu$ and $\nu$.

\begin{defn}\label{def:order}
Let $q$ and $n$ be natural numbers with $q\ge 2$. Let $\lambda\in\mathcal{P}(qn)$. We can uniquely write $\lambda$ as $\lambda=(\mu\circ \nu)^\star$ where $\mu$ is the partition consisting of all the parts of $\lambda$ that are not divisible by $q$ and $\nu$ is the partition consisting of all the parts of $\lambda$ that are multiples of $q$. 
In particular we have that $$\lambda=\big((k_1q+x_1,\ldots, k_tq+x_t)\circ (r_1q,\ldots ,r_sq)\big)^\star,\qquad(*)$$
where $k_1\geq k_2\geq \cdots\geq k_t\geq 0$, $r_1\geq r_2\geq \cdots\geq r_s> 0$ and where $x_j\in\{1,\ldots, q-1\}$ for all $j\in\{1,\ldots, t\}$.
Since $\lambda\vdash qn$ there exists $\zeta_q(\lambda)\in\mathbb{N}_{0}$ such that $x_1+x_2+\cdots + x_t=\zeta_q(\lambda)q$.
Notice that $\zeta_q(\lambda)=\frac{1}{q}(x_1+\cdots + x_t)\leq t\cdot \frac{q-1}{q}\leq t$.
We denote by $A_\lambda$ the multiset of $q$-residues $\{x_1,\ldots, x_t\}$. We define a total order $\succ$ on the indexing set $\{1,2,\ldots, t\}$ as follows. Let $i,j$ be distinct elements in $\{1,\ldots, t\}$. If $x_i>x_j$ then $i\succ j$. When $x_i=x_j$ then we let $i\succ j$ if and only if $i>j$. 

We denote by $\lambda_{\succ}$ the composition
$$\lambda_{\succ}=(k_{i_1}q+x_{i_1},k_{i_2}q+x_{i_2}\ldots, k_{i_t}q+x_{i_t}),$$
where $i_1,\ldots, i_t\in\{1,\ldots, t\}$ are such that $i_1\succ i_2\succ \cdots\succ i_t$. 
\end{defn}

\begin{defn}\label{def:omega}
Given $\lambda$ as in equation $(*)$ above we let $\Omega_q(\lambda)_{\succ}$ be the composition  defined by 
$\Omega_q(\lambda)_{\succ}=\big(\lambda_{\succ}-(k_{i_1}+1,\ldots, k_{i_{\zeta_q(\lambda)}}+1, k_{i_{\zeta_q(\lambda)}},\ldots, k_{i_t})\big).$
Moreover we denote by $\Omega_q(\lambda)$
the partition of $(q-1)n$ defined by 
$\Omega_q(\lambda)=\big[\Omega_q(\lambda)_{\succ}\circ (r_1(q-1), \ldots r_s(q-1))\big]^\star.$
\end{defn}

The partition $\Omega_q(\lambda)$ should be thought of as being obtained from $\lambda$ by multiplying each part of $\lambda$ by $\tfrac{q-1}{q}$, with appropriate rounding. In particular, for all $i$ we have that $\lambda_i-\Omega_q(\lambda)_i \in\{\lfloor\tfrac{\lambda_i}{q}\rfloor, \lfloor\tfrac{\lambda_i}{q}\rfloor+1 \}$, and so $\tfrac{q-1}{q}\lambda_i-1<\Omega_q(\lambda)_i<\tfrac{q-1}{q}\lambda_i+1$.

\begin{exmp}
Let $q=3$, $n=16$ and let $\lambda=(9,8,7,7,6,4,4,3)\vdash 48$. Then we have that $\zeta_3(\lambda)=2$ and that $\Omega_3(\lambda)=(6,5,5,5,4,3,2,2)$.
\end{exmp}

\begin{lem}\label{lem:skew}
Let $q,n\in\mathbb{N}$ be such that $q\ge 2$ and $n\ge 5$. Let $\lambda\vdash qn$. Then $$[\lambda\setminus\Omega_q(\lambda)]\notin\{[(n-1,1)], [(n-1,1)]^\circ\}.$$
\end{lem}

\begin{proof}
Let $\Omega:=\Omega_q(\lambda)$. Suppose for a contradiction that $[\lambda\setminus\Omega]\cong[(n-1,1)]$. Then there exists $j\in\mathbb{N}$ such that $\lambda_j-\Omega_j = n-1$, $\lambda_{j+1}-\Omega_{j+1}=1$ and $\Omega_j=\Omega_{j+1}$. Thus $\lambda_j\ge(n-2)q$ and $\lambda_{j+1}<2q$. It follows that $\Omega_j\geq (n-2)(q-1)> 2(q-1)\geq \Omega_{j+1}$, which is a contradiction.
Similarly, let us now assume that $[\lambda\setminus\Omega]\cong[(n-1,1)]^\circ$. Then there exists $j\in\mathbb{N}$ such that $\lambda_j-\Omega_j = 1$, $\lambda_{j+1}-\Omega_{j+1}=n-1$ and $\lambda_j=\lambda_{j+1}$. It follows that $\lambda_j<2q$ whilst $\lambda_{j+1}\ge(n-2)q$, which is a contradiction.
\end{proof}

We will also need the following technical lemma. 

\begin{lem}\label{lem:q=2-not-omega}
Let $p^k\notin\{ 3, 5\}$ and let $\lambda=(\lambda_1,\dotsc,\lambda_t)\in A(2,p^k)$. Suppose that $t,\lambda_1<p^k$. If $\Omega_2(\lambda)\in\{(3,1^{p^k-3}), (2^2,1^{p^k-4})\}$ then $\lambda\in\mathcal{D}(2,p^k)$ 
\end{lem}

\begin{proof}
The key idea in this proof is that there are very few partitions $\lambda\in\mathcal{P}(2p^k)$ such that $\Omega=\Omega_2(\lambda)\in\{(3,1^{p^k-3}), (2^2,1^{p^k-4})\}$. 
We list all of them below. Moreover, for each of these partitions we explicitly exhibit a pair $\gamma, \delta\in\Delta(p^k)$ such that $\gamma\neq \delta$ and such that $c_{\gamma,\delta}^\lambda\neq 0$. 
Clearly this implies that $\lambda\in\mathcal{D}(2,p^k)$.

Suppose that $\Omega=(3,1^{p^k-3})$, then $t=p^k-2$ or $p^k-1$. By Definition \ref{def:omega}, $1\le\lambda_t\le \lambda_{p^k-2}\le 3$ and $5\le\lambda_1\le 7$. 
If $\lambda_t=3$, or $\lambda_t=2$ and $t=p^k-1$, then $|\lambda|>2p^k$ which is a contradiction. 
If $\lambda_t=2$ and $t=p^k-2$ then $\lambda=(5,3,2^{p^k-4})$ or $(6,2^{p^k-3})$.
Now assume that $\lambda_t=1$.
If $t=p^k-1$ and $\lambda_{t-2}\ge 2$, then $\lambda=(6,2^{p^k-4},1^2)$, $(5,3,2^{p^k-5},1^2)$ or $(5,2^{p^k-3},1)$.
If $t=p^k-1$ and $\lambda_{t-2}=1$, or $t=p^k-2$, then by Definition~\ref{def:omega} we have $\Omega_{p^k-2}=0$, which is a contradiction.
	
In the table below we list the aforementioned possibilities for $\lambda$ (first column) and exhibit a pair $\gamma, \delta\in\Delta(p^k)$ such that $\gamma\neq \delta$ and such that $c_{\gamma,\delta}^\lambda\neq 0$ (second column). 

\begin{center}
\begin{tabular}{|c|c|}
\hline
$\lambda$ & $\gamma,\delta$\\
\hline
\hline
$(5,3,2^{p^k-4})$ & \multirow{2}{*}{$(3,2,1^{p^k-5}),\ (4,1^{p^k-4})$}\\
\cline{1-1}
$(6,2^{p^k-3})$ & \\
\hline
$(6,2^{p^k-4},1^2)$ & $(4,1^{p^k-4}),\ (3,1^{p^k-3})$\\
\hline
$(5,3,2^{p^k-5},1^2)$ & \multirow{2}{*}{$(3,2,1^{p^k-5}),\ (3,1^{p^k-3})$}\\
\cline{1-1}
$(5,2^{p^k-3},1)$ & \\
\hline
\end{tabular}
\end{center}

\medskip

Now let $\Omega=(2^2,1^{p^k-4})$, then $t=p^k-2$ or $p^k-1$. By Definition \ref{def:omega}, $1\le\lambda_t\le 3$ and $3\le\lambda_1,\lambda_2\le 5$. 
If $\lambda_t=3$ then $|\lambda|>2p^k$, a contradiction. 
If $\lambda_t=2$ then $\lambda=(3,3,2^{p^k-3})$, $(5,3,2^{p^k-4})$, $(4,4,2^{p^k-4})$, $(4,3,3,2^{p^k-5})$ or $(3^4,2^{p^k-6})$. 
Finally, suppose that $\lambda_t=1$. 
If $t=p^k-1$ and $\lambda_{t-2}\ge 2$, then $\lambda=(5,3,2^{p^k-5},1^2)$, $(4,4,2^{p^k-5},1^2)$, $(4,3,3,2^{p^k-6},1^2)$, $(3^4,2^{p^k-7},1^2)$, $(4,3,2^{p^k-4},1)$ or $(3^3,2^{p^k-5},1)$. 
If $t=p^k-1$ and $\lambda_{t-2}=1$, or $t=p^k-2$, then by Definition~\ref{def:omega} we have $\Omega_{p^k-2}=0$, which is a contradiction.
	
As done for the previous case, in the tables below we list the aforementioned possibilities for $\lambda$ and exhibit a pair $\gamma, \delta\in\Delta(p^k)$ such that $\gamma\neq \delta$ and such that $c_{\gamma,\delta}^\lambda\neq 0$. 
	
\begin{center}
\begin{tabular}{|c|c|}
\hline
$\lambda$ & $\gamma,\delta$\\
\hline
\hline
$(3,3,2^{p^k-3})$ & \multirow{2}{*}{$(2^2,1^{p^k-4}),\ (3,2,1^{p^k-5})$}\\
\cline{1-1}
$(5,3,2^{p^k-4})$ & \\
\hline
$(4,4,2^{p^k-4})$ & \multirow{2}{*}{$(3,2,1^{p^k-5}),\ (4,2,1^{p^k-6})$}\\
\cline{1-1}
$(4,3,3,2^{p^k-5})$ & \\
\hline
$(3^4,2^{p^k-6})$ & $(3,2,1^{p^k-5}),\ (3,2,2,1^{p^k-7})$\\
\hline
\end{tabular}
\hspace{10pt}
\begin{tabular}{|c|c|}
\hline
$\lambda$ & $\gamma,\delta$\\
\hline
\hline
$(5,3,2^{p^k-5},1^2)$ & \multirow{5}{*}{$(3,2,1^{p^k-5}),\ (3,1^{p^k-3})$}\\
\cline{1-1}
$(4,4,2^{p^k-5},1^2)$ & \\
\cline{1-1}
$(4,3,3,2^{p^k-6},1^2)$ & \\
\cline{1-1}
$(4,3,2^{p^k-4},1)$ & \\
\cline{1-1}
$(3^3,2^{p^k-5},1)$ & \\
\hline
$(3^4,2^{p^k-7},1^2)$ & $(3,2,1^{p^k-5}),\ (2^3,1^{p^k-6})$\\
\hline
\end{tabular}
\end{center}

Thus in all cases, $\lambda\in\mathcal{D}(2,p^k)$.
\end{proof}

We now turn to the proof of Proposition~\ref{prop:have-property-star}. The proof is done by induction on $q$ for each fixed $p$ and $k$. For the remainder of this section, fix $p$ an odd prime and $k\in\mathbb{N}$ such that $p^k\ne 3,5$. We begin with the case where $q=2$.

\begin{prop}\label{prop:q=2}
Let $p^k\ne 3,5$. Then $\mathcal{D}(2,p^k)=A(2,p^k)$.
\end{prop}

\begin{proof}
It is easy to check that $\mathcal{D}(2,p^k)\subseteq A(2,p^k)$. 
Hence we now let $\lambda=(\lambda_1,\lambda_2,\dotsc,\lambda_t)\in A(2,p^k)$ and we aim to show that $\lambda\in\mathcal{D}(2,p^k)$. 
To ease the notation, we let $\chi^\lambda\down:=\chi^\lambda\down_{\fS_{p^k}\times\fS_{p^k}}$ for the rest of this proof. First suppose $t\ge p^k$ and consider the partition $\mu$ obtained from $\lambda$ by reducing the length of each of the last $p^k$ parts of $\lambda$ by one. More precisely we  have $$\mu=(\lambda_1,\dotsc,\lambda_{t-p^k}, \lambda_{t-p^k+1}-1,\lambda_{t-p^k+2}-1,\dotsc,\lambda_t-1)\vdash p^k.$$
Then $(1^{p^k})\in\mathcal{LR}([\lambda\setminus\mu])$. If $\mu\in\Delta(p^k)\setminus\{(1^{p^k})\}$ then $\lambda\in\mathcal{D}(2,p^k)$. 
Otherwise, $\mu\in\{(1^{p^k}), (2,1^{p^k-2}), (p^k-1,1)\}$. 
If $\mu=(1^{p^k})$ then by inverting the process used to construct $\mu$
from $\lambda$,
we deduce that $\lambda\in\{(1^{2p^k}), (2^{p^k})\}.$ But this would imply that $\lambda\notin A(2,p^k)$, which is a contradiction. Thus $\mu$ cannot be equal to $(1^{p^k})$. 
In the following table, we consider the remaining possibilities for $\mu$ (first column) and list the consequent possibilities for $\lambda$ (second column). When a resulting candidate for $\lambda$ lies in $A(2,p^k)$, we exhibit $\gamma,\delta\in\Delta(p^k)$ such that $\gamma\ne\delta$ and $\chi^\gamma\times\chi^\delta$ is an irreducible constituent of $\chi^\lambda\down$ (third column), and hence we deduce that $\lambda\in\mathcal{D}(2,p^k)$. 

\begin{center}
\begin{tabular}{|c|cc|c|}
\hline
$\mu$ & $\lambda$ & & $\gamma,\delta$\\
\hline
\hline
\multirow{2}{*}{$(1^{p^k})$} & $(1^{2p^k})$ & $\notin A_p(k,2)$ & $-$\\
\cline{2-4}
& $(2^{p^k})$ & $\notin A_p(k,2)$ & $-$\\
\hline
\multirow{3}{*}{$(2,1^{p^k-2})$} & $(2,1^{2p^k-2})$ & $\notin A_p(k,2)$ & $-$\\
\cline{2-4}
& $(2^{p^k-1},1^2)$ & $\in A_p(k,2)$ & $(2^2,1^{p^k-4}),\ (2^3,1^{p^k-6})$\\
\cline{2-4}
& $(3,2^{p^k-2},1)$ & $\in A_p(k,2)$ & $(3,1^{p^k-3}),\ (2^2,1^{p^k-4})$\\
\hline
\multirow{3}{*}{$(p^k-1,1)$} & $(p^k-1,1^{p^k+1})$ & $\in A_p(k,2)$ & $(p^k-2,1^2),\ (1^{p^k})$\\
\cline{2-4}
& $(p^k-1,2,1^{p^k-1})$ & $\in A_p(k,2)$ & $(p^k-2,2),\ (1^{p^k})$\\
\cline{2-4}
& $(p^k,2,1^{p^k-2})$ & $\in A_p(k,2)$ & $(p^k-2,2),\ (3,1^{p^k-3})$\\
\hline
\end{tabular}
\end{center}

\medskip

In light of the above discussion we may now assume that $t<p^k$. By Lemma~\ref{lem:star-conj}, we may also assume $\lambda_1<p^k$.
Let $\Omega=\Omega_2(\lambda)\vdash p^k$ and let $\gamma=[\lambda\setminus\Omega]$. If $\Omega\in\Delta(p^k)$ and there exists some $\nu\in\mathcal{LR}(\gamma)$ with $\nu\in\Delta(p^k)\setminus\{\Omega\}$, then $\lambda\in\mathcal{D}(2,p^k)$. Otherwise, either $$\Omega\in\{(p^k-1,1), (2,1^{p^k-2})\},\ \  \text{or}\ \ \Omega\in\Delta(p^k)\ \ \text{and}\ \  \mathcal{LR}(\gamma)\subseteq\{(p^k-1,1),(2,1^{p^k-2}),\Omega\}.$$ Since $\lambda_1<p^k$ then $\Omega\neq (p^k-1,1)$. 
If $\Omega=(2,1^{p^k-2})$ then we immediately deduce that $t=p^{k-1}$ and $\lambda_t\in\{2,3\}$, by Definition \ref{def:omega}. If $\lambda_t=3$ then $|\lambda|>2p^k$, a contradiction. Therefore $\lambda_t=2$ and thus $\lambda\in\{(3^2,2^{p^k-3}), (4,2^{p^k-2})\}$, a subset of $\mathcal{D}(2,p^k)$ by direct verification. Hence we can assume $\Omega\in\Delta(p^k)$ and $\mathcal{LR}(\gamma)\subseteq\{(p^k-1,1),(2,1^{p^k-2}),\Omega\}$.

First suppose $(p^k-1,1)\in\mathcal{LR}(\gamma)$. Then $\gamma\notin\{[(p^k-1,1)],[(p^k-1,1)]^\circ\}$ by Lemma~\ref{lem:skew}. Therefore, there exists $\nu\in\mathcal{LR}(\gamma)\cap\{(p^k),(p^k-2,1,1),(p^k-2,2)\}$ by Lemma~\ref{lem:filling-type-(m-1,1)}. 
Since $\lambda_1<p^k$ and $p^k>5$ we deduce (from Definition \ref{def:omega}) that $\Omega\notin\{(p^k),(p^k-2,1,1),(p^k-2,2)\}$.
Thus $\lambda\in\mathcal{D}(2,p^k)$ since $\Omega\ne\nu$ and $\chi^\Omega\times\chi^\nu$ is an irreducible constituent of $\chi^\lambda\down$.

Next suppose $(2,1^{p^k-2})\in\mathcal{LR}(\gamma)$. By Lemma~\ref{lem:filling}, $\gamma$ must have at least $p^k-1$ non-empty rows, so $t \ge p^k-1$. Thus $t=p^k-1$ and $\lambda_i-\Omega_i=2$ for a unique $i\in\{1,2,\dotsc,t\}$ while $\lambda_j-\Omega_j=1$ for all $j\ne i$, since $|\gamma|=p^k$. If $\gamma\cong[(2,1^{p^k-2})]$ then $\Omega=(u^{p^k-1})$ for some $u\in\mathbb{N}$. But then $2p^k=|\lambda|=u(p^k-1)+p^k$, a contradiction. Also, if $\gamma^\circ\cong[(2,1^{p^k-2})]$ then $\lambda=(u^{p^k-1})$ for some $u\in\mathbb{N}$. But then $u(p^k-1)=2p^k$, a contradiction since $p^k>3$. 
Hence $\gamma\notin\{[(2,1^{p^k-2})], [(2,1^{p^k-2})]^\circ\}$. 
It is now easy to see that there exists some $\nu\in\mathcal{LR}(\gamma)\cap \{(2^2,1^{p^k-4}), (3,1^{p^k-3})\}$. Thus $\lambda\in\mathcal{D}(2,p^k)$: by Lemma~\ref{lem:q=2-not-omega} if $\Omega=\nu$, or by considering the constituent $\chi^\Omega\times\chi^\nu$ if $\Omega\neq \nu$.

\medskip

Finally we may assume that $\mathcal{LR}(\gamma)=\{\Omega\}$, and so either $\gamma\cong[\Omega]$ or $\gamma^\circ\cong[\Omega]$, by Lemma~\ref{lem:BK}. If $\gamma\cong[\Omega]$ then we deduce from Definition \ref{def:omega} that $\lambda$ is a rectangle of even width and odd height. More precisely, 
$\lambda=((2p^b)^{p^a})$ for some $a+b=k$, with $a,b\ge 1$ since $\lambda\in A(2,p^k)$.
Let $\mu=((2^{p^b})^c, d+1, 1^{p^a-c-1})$ where $p^k-p^a=c(2p^b-1)+d$ and $0\le d<2p^b-1$. Note that $c\ge 1$ and $p^a\ge c+1$. In particular, $\mu_1=\lambda_1$ and $\mu'_1=\lambda'_1$, and $[\lambda\setminus\mu]^\circ\cong[\nu]$ for some partition $\nu$. Hence $\mathcal{LR}([\lambda\setminus\mu])=\{\nu\}$ by Lemma~\ref{lem:BK}. Since $\nu_1=2p^b-1$, we have that $\nu\notin\{\mu,(p^k-1,1),(2,1^{p^k-2})\}$. Therefore, $\chi^\mu\times\chi^\nu$ is a constituent of $\chi^\lambda\down$ such that $\mu\ne\nu$ and $\mu,\nu\in\Delta(p^k)$, and thus $\lambda\in\mathcal{D}(2,p^k)$.
Finally, if $\gamma^\circ\cong[\Omega]$ then we find similarly that  the top row of $\gamma$ is row 1 (since $\Omega\ne(1^{p^k})$). Thus $\lambda$ is again a rectangle. It now suffices to consider $\lambda$ of even height and odd width, but then $\lambda\in\mathcal{D}(2,p^k)$ by Lemma~\ref{lem:star-conj}.
\end{proof}

We are now ready to show that Proposition \ref{prop:have-property-star} holds. 


\begin{proof}[Proof of Proposition \ref{prop:have-property-star}]
We proceed by induction on $q$. The base step $q=2$ follows from Proposition \ref{prop:q=2}. We now fix $q\geq 3$ and assume that $A(q-1,p^k)=\mathcal{D}(q-1,p^k)$.
Let $\lambda\in A(q, p^k)$. Our aim is to show that there always exists a pair $(\mu, \nu)\in A(q-1,p^k)\times \Delta(p^k)$ such that $\chi^\mu\times \chi^\nu$ is an irreducible constituent of the restriction of $\chi^\lambda$ to $\fS_{(q-1)p^k}\times \fS_{p^k}$. Using the inductive hypothesis this immediately implies that $\lambda\in\mathcal{D}(q,p^k)$.

Let $\lambda=(\lambda_1,\dotsc,\lambda_t)\in A(q,p^k)$ and let $m:=(q-1)p^k$. First suppose $t\ge p^k$ and let $$\mu=(\lambda_1,\dotsc,\lambda_{t-p^k}, \lambda_{t-p^k+1}-1,\lambda_{t-p^k+2}-1,\dotsc,\lambda_t-1)\vdash m.$$
Clearly $(1^{p^k})\in\mathcal{LR}([\lambda\setminus\mu])$. Hence if $\mu\in A(q-1,p^k)$ then $\lambda\in\mathcal{D}(q,p^k)$, as remarked above. 
We now analyse case by case the few possible situations where $\mu\notin A(q-1,p^k)$.
If $\mu=(m)$ then by inverting the process used to construct $\mu$
from $\lambda$, we deduce that $\lambda\in\{(m,1^{p^k}), (m+1,1^{p^k-1})\}.$ 
In both cases we have that $$\underbrace{\chi^{(p^k)}\times\cdots\times\chi^{(p^k)}}_{q-1}\times\chi^{(1^{p^k})}\ \ \text{is an irreducible constituent of}\ \ (\chi^\lambda)\down_{\fS_{p^k}^{\times q}}.$$
Therefore $\lambda\in\mathcal{D}(q,p^k)$.
In the following table, we consider the remaining possibilities for $\mu$ (first column) and list the consequent possibilities for $\lambda$ (second column). When a resulting candidate for $\lambda$ lies in $A(q,p^k)$, we exhibit $\mu_1,\dotsc,\mu_q\in\Delta(p^k)$ such that $\mu_i$ are not all equal and such that $\chi^{\mu_1}\times\cdots\times\chi^{\mu_q}$ is a constituent of $\chi^\lambda\down_{\fS_{p^k}^{\times q}}$ (third column). Hence we deduce that $\lambda\in\mathcal{D}(q,p^k)$.

\medskip

\begin{center}
\begin{tabular}{|c|cc|c|}
\hline
$\mu$ & $\lambda$ & & $\mu_1,\dotsc,\mu_q$\\
\hline
\hline
\multirow{2}{*}{$(m)$} & $(m,1^{p^k})$ & $\in A_p(k,q)$ & \multirow{2}{*}{$\underbrace{(p^k),\dotsc,(p^k)}_{q-1},\ (1^{p^k})$}\\[5pt]
\cline{2-3}
& $(m+1,1^{p^k-1})$ & $\in A_p(k,q)$ & \\[5pt]
\hline
\multirow{3}{*}{} & $(m-1,1^{p^k+1})$ & $\in A_p(k,q)$ & $(p^k-2,1,1),\ \underbrace{(p^k),\dotsc,(p^k)}_{q-2},\ (1^{p^k})$\\[5pt]
\cline{2-4}
$(m-1,1)$ & $(m-1,2,1^{p^k-1})$ & $\in A_p(k,q)$ & $(p^k-2,2),\ \underbrace{(p^k),\dotsc,(p^k)}_{q-2},\ (1^{p^k})$\\[5pt]
\cline{2-4}
& $(m,2,1^{p^k-2})$ & $\in A_p(k,q)$ & $(p^k-2,2),\ \underbrace{(p^k),\dotsc,(p^k)}_{q-2},\ (3,1^{p^k-3})$ \\[5pt]
\hline
$(2,1^{m-2})$ & $(2,1^{qp^k-2})$ & $\notin A_p(k,q)$ & $-$ \\
\hline
$(1^m)$ & $(1^{qp^k})$ & $\notin A_p(k,q)$ & $-$ \\
\hline
\multirow{2}{*}{$(p^k,p^k)\ \mathrm{if}\ q=3$} & $(p^k,p^k,1^{p^k})$ & $\in A_p(k,3)$ & \multirow{2}{*}{$(p^k),\ (p^k),\ (1^{p^k})$}\\
\cline{2-3}
& $(p^k+1,p^k+1,1^{p^k-2})$ & $\in A_p(k,3)$ & \\
\hline
\multirow{2}{*}{$(2^{p^k})\ \mathrm{if}\ q=3$} & $(2^{p^k},1^{p^k})$ & $\in A_p(k,3)$ & \multirow{2}{*}{$(2^2,1^{p^k-4}),\ (2^2,1^{p^k-4}),\ (1^{p^k})$}\\
\cline{2-3}
& $(3^{p^k})$ & $\in A_p(k,3)$ & \\
\hline
\end{tabular}
\end{center}

\medskip

Now we may assume $t<p^k$. By Lemma~\ref{lem:star-conj}, we may also assume $\lambda_1<p^k$.
Consider $\Omega=\Omega_q(\lambda)\vdash m$ and let $\gamma=[\lambda\setminus\Omega]$. Since $[\Omega]\subset[\lambda]$, the above assumptions imply $\Omega\in A(q-1,p^k)$. Therefore if $\mathcal{LR}(\gamma)\cap\Delta(p^k)\ne\emptyset$, then $\lambda\in\mathcal{D}(q,p^k)$ by induction. 
So suppose $\mathcal{LR}(\gamma)\subseteq \{(p^k-1,1),(2,1^{p^k-2})\}$.

If $(p^k-1,1)\in\mathcal{LR}(\gamma)$ then $\mathcal{LR}(\gamma)\cap\{(p^k),(p^k-2,2),(p^k-2,1,1)\}\ne\emptyset$, by Lemmas~\ref{lem:skew} and~\ref{lem:filling-type-(m-1,1)}. This is clearly a contradiction. 
Thus $\mathcal{LR}(\gamma)=\{(2,1^{p^k-2})\}$ and so by Lemma~\ref{lem:BK} we have that $\gamma\cong[(2,1^{p^k-2})]$ or $\gamma^\circ\cong[(2,1^{p^k-2})]$.

If $\gamma\cong[(2,1^{p^k-2})]$, then $t=p^k-1$ and $\Omega=(u^{p^k-1})$ for some $u\in\mathbb{N}$. But then $qp^k=|\lambda|=u(p^k-1)+p^k$ implies that $p^k$ divides $u$. Since $q\le p$, we must have $k=1$ and $q=p=u$, giving $\lambda=(p+2,(p+1)^{p-2})\vdash p^2$. Since $$\chi^{(p-2,1,1)}\times\chi^{(5,1^{p-5})}\times(\chi^{(p)})^{\times (p-2)}\ \ \text{is an irreducible constituent of}\ \ \chi^\lambda\down_{\fS_{p^k}^{\times p}},$$ we have that $\lambda\in\mathcal{D}(p,p)$. By a similar argument, we see that if $\gamma^\circ\cong[(2,1^{p^k-2})]$, then necessarily we have $k=1$, $q=p-1$ and $\lambda=(p^{p-1})\in\mathcal{D}(p-1,p)$.

Thus $A(q,p^k)\subseteq\mathcal{D}(q,p^k)$. Since it is easy to see from the Littlewood\textendash Richardson rule that $(p^{k+1})$, $(p^{k+1}-1,1)$, $(2,1^{p^{k+1}-2})$ and $(1^{p^{k+1}})$ do not belong to $\mathcal{D}(q,p^k)$, then $A(q,p^k)=\mathcal{D}(q,p^k)$ as claimed.
\end{proof}

As a corollary, we can now show that Proposition \ref{prop:eugenew} holds. This was the original main goal of this section.

\begin{proof}[Proof of Proposition \ref{prop:eugenew}]
The statement follows from Proposition \ref{prop:have-property-star} by setting $q=p$, whenever $p^k\neq 5$. If $p^k=5$ then direct verification shows that $\mathcal{D}(5,5)=A(5,5)$.
\end{proof}


\section{Proof of Theorem A}\label{sec:pf-thm-A}

In this final section we prove Theorem A of the introduction for all odd primes $p$ and all natural numbers $n$. First, we prove Theorem A when $p\geq 5$.

\begin{prop}\label{prop:p>3-thmA}
Let $p\ge 5$ be an odd prime and $n\in\mathbb{N}$. Let $\lambda\vdash n$. Then $\innprod{\lambda}{P_n} = 0$ if and only if $n=p^k$ for some $k\in\mathbb{N}$ and $\lambda\in\{(p^k-1,1), (2,1^{p^k-2}) \}$.
\end{prop}

\begin{proof}
Let $\Sigma(n)$ denote the sum of the $p$-adic digits of $n$, that is, the sum of the digits when $n$ is expressed in base $p$. We prove the assertion by induction on $\Sigma(n)$, with Theorem~\ref{thm:n-prime-power} and Lemma~\ref{lem:n<p} providing the base case $\Sigma(n)=1$. 

Now assume that $n>p$ and that $\Sigma(n)\ge 2$. 
Let $p^k$ be the largest $p$-adic digit of $n$ and let $m=n-p^k$. Clearly $k>0$ and $\Sigma(m)=\Sigma(n)-1$.
For any pair $(\mu,\nu)\in\mathcal{P}(m)\times\mathcal{P}(p^k)$ we say that $(\mu,\nu)$ is a \textit{suitable pair} for $\lambda\in\mathcal{P}(n)$ if 
$$c_{\mu,\nu}^\lambda\neq 0\ \ \text{and}\ \ \innprod{\mu}{P_m}\cdot\innprod{\nu}{P_{p^k}}> 0.$$
We denote by $\mathcal{S}(\lambda)$ the set of suitable pairs for $\lambda$.
It is clear that if $\mathcal{S}(\lambda)\neq\emptyset$ then $\innprod{\lambda}{P_n}> 0$, since $P_n\cong P_m\times P_{p^k}$. We will now show that $\mathcal{S}(\lambda)\neq\emptyset$ for all $\lambda\in\mathcal{P}(n)$.

\smallskip

First suppose that $\Sigma(m)>1$ and let $\lambda\in\mathcal{P}(n)$. Theorem \ref{thm:n-prime-power} together with the inductive hypothesis shows that $$\mathcal{S}(\lambda)=\{(\mu,\nu)\in\mathcal{P}(m)\times \Delta(p^k)\ |\ c_{\mu,\nu}^\lambda\neq 0\}.$$ 
If $\lambda_2\geq 2$ then there exists $\nu\in\Delta(p^k)$ such that $[\nu]\subseteq [\lambda]$. Hence $\mathcal{LR}([\lambda\setminus\nu])\times \{\nu\}\subseteq \mathcal{S}(\lambda)\neq \emptyset$.
Otherwise $\lambda$ is a hook partition. Since $|\lambda|>p^k$, there exists some hook partition $\nu\notin\{(p^k-1,1),(2,1^{p^k-2})\}$ such that $[\nu]\subset[\lambda]$. Therefore again we have $\mathcal{LR}([\lambda\setminus\nu])\times \{\nu\}\subseteq \mathcal{S}(\lambda)\neq \emptyset$.

Now we may assume that $\Sigma(m)=1$, that is, $m=p^l\leq p^k$ for some integer $l$. 
If $l=k$ and $\lambda\notin A(2,p^k)$ then $\mathcal{S}(\lambda)$ contains $((p^k), (p^k))$ or $((1^{p^k}),(1^{p^k}))$ and so is non-empty. Otherwise, $\lambda\in A(2, p^k)=\mathcal{D}(2,p^k)$ and we deduce that $\mathcal{S}(\lambda)\neq \emptyset$ by Proposition~\ref{prop:have-property-star}. 
Suppose finally that $k>l$.
As above, $|\lambda|>p^l$ implies that there exists some $\mu\in\Delta(p^l)$ such that $[\mu]\subset[\lambda]$. Let $\nu\in\mathcal{LR}([\lambda\setminus\mu])$. If $\nu\in\Delta(p^k)$ then $(\mu, \nu)\in\mathcal{S}(\lambda)\neq \emptyset$.
Otherwise, $\mathcal{LR}([\lambda\setminus\mu])\subseteq \{(p^k-1,1),(2,1^{p^k-2})\}$. By Lemmas~\ref{lem:filling-type-(m-1,1)} and~\ref{lem:BK}, we must have $$[\lambda\setminus\mu]\in\{[(p^k-1,1)], [(p^k-1,1)]^\circ, [(2,1^{p^k-2})],[(2,1^{p^k-2})]^\circ\}.$$
Since $k>l$ we can immediately rule out $[(p^k-1,1)]^\circ$ and $[(2,1^{p^k-2})]^\circ$. Hence if $\mu=(\mu_1,\dotsc,\mu_s)$, we must have either (a) $\lambda=(\mu_1+p^k-1,\mu_2+1,\mu_3,\dotsc,\mu_s)$ and $\mu_1=\mu_2$, or (b) $\lambda=(\mu_1,\dotsc,\mu_s,2,1^{p^k-2})$ and $\mu_s\ge 2$. However, any partition satisfying (b) is conjugate to a partition satisfying (a), so by Lemma~\ref{lem:conj} it remains to consider (a).
In this case, we let $\tilde{\mu}=(\mu_1+1,\mu_2,\dotsc,\mu_{s-1},\mu_s-1)$. Hence we have $(p^k)\in\mathcal{LR}([\lambda\setminus\tilde{\mu}])$. Moreover, $\mu_1=\mu_2$ implies that $\tilde{\mu}\ne(p^l-1,1)$. Hence $(\tilde\mu,(p^k))\in\mathcal{S}(\lambda)$ unless $\tilde{\mu}=(2,1^{p^l-2})$. But in this case we would have $\mu=(1^{p^l})$, $\lambda=(p^k,2,1^{p^l-2})$ and therefore $((2^2,1^{p^l-4}),(p^k-2,1,1))\in\mathcal{S}(\lambda)\neq \emptyset$.

Thus in all instances we have found a suitable pair for $\lambda\vdash n$, and hence $\innprod{\lambda}{P_n}>0$.
\end{proof}

Finally we conclude this article by verifying Theorem A for $p=3$. 

\begin{prop}\label{prop:p=3-thmA}
Let $p=3$. Then $\innprod{\lambda}{P_n} = 0$ if and only if $n=3^k$ for some $k\in\mathbb{N}$ and $\lambda\in\{(3^k-1,1), (2,1^{3^k-2}) \}$, or $n\leq 10$ and $\lambda$ is one of the following partitions: $$(2,2);\ \ (3,2,1);\ \ (5,4), (2^4,1), (4,3,2), (3^2,2,1);\ \ (5,5),(2^5).$$
\end{prop}

\begin{proof}
The same argument as in the proof of Proposition~\ref{prop:p>3-thmA} shows that the assertion holds for all $n\in\mathbb{N}$ divisible by 27. Since the assertion may be verified computationally for $n\le 27$, it remains to consider $n$ of the form $27t+u$ where $t,u\in\mathbb{N}$ and $u<27$.

Given $\lambda\vdash n$, it is clear that there exists $\mu\vdash 27t$ (and if $27t$ is a power of 3, say $3^k$, we can further choose $\mu$ such that $\mu\in\Delta(3^k)$) such that $[\mu]\subset[\lambda]$. Let $\nu\in\mathcal{LR}([\lambda\setminus\mu])$. If $u\notin\{3,4,6,9,10\}$, then
$$ \innprod{\lambda}{P_n} \ge \langle \chi^\mu\times\chi^\nu\down_{P_{27t}\times P_u}, \triv_{P_{27t}\times P_u}\rangle = \innprod{\mu}{P_{27t}}\cdot\innprod{\nu}{P_u} > 0.$$

Now assume $u\in\{3,4,6,9,10\}$ and let $T(3)=\{(2,1)\}$, $T(4)=\{(2,2)\}$, $T(6)=\{(3,2,1)\}$, $T(9)=\{(8,1),(5,4),(4,3,2),(3^2,2,1),(2^4,1),(2,1^7)\}$ and $T(10)=\{(5,5),(2^5)\}$. Again there exists $\mu\in\mathcal{P}(u)\setminus T(u)$ satisfying $[\mu]\subset[\lambda]$, so if $27t$ is not a power of 3 then we may take any $\nu\in\mathcal{LR}([\lambda\setminus\mu])$ to see that $\innprod{\lambda}{P_n}\ge \innprod{\mu}{P_u}\cdot\innprod{\nu}{P_{27t}}>0$.

Otherwise, let $27t=3^k$ for some $k\ge 3$. The above argument applies unless $\mathcal{LR}([\lambda\setminus\mu])\subseteq \{(3^k-1,1),(2,1^{3^k-2})\}$. By Lemmas~\ref{lem:filling-type-(m-1,1)} and~\ref{lem:BK}, $$[\lambda\setminus\mu]\in\{[(3^k-1,1)], [(3^k-1,1)]^\circ, [(2,1^{3^k-2})], [(2,1^{3^k-2})]^\circ\}.$$ But $u\le 10$, so letting $\mu=(\mu_1,\dotsc,\mu_s)$, we must have either (a) $\lambda=(\mu_1+3^k-1,\mu_2+1,\mu_3,\dotsc,\mu_s)$ and $\mu_1=\mu_2$, or (b) $\lambda=(\mu_1,\dotsc,\mu_s,2,1^{p^3-2})$ and $\mu_s\ge 2$. However, any partition satisfying (b) is conjugate to a partition satisfying (a), so by Lemma~\ref{lem:conj} it remains to consider (a). If $\mu=(1^u)$, then $\lambda=(3^k,2,1^{u-2})$ and $\chi^\lambda\down_{P_{3^k}\times P_u}$ has $\chi^{(3^k-2,2)}\times\chi^{(3,1^{u-3})}$ as a constituent. Otherwise, $\chi^\lambda\down_{P_{3^k}\times P_u}$ has $\chi^{\gamma}\times\chi^{(u)}$ as a constituent where $\gamma=(u_1+3^k-u,\mu_2,\dotsc,\mu_s)\in\Delta(3^k)$.
\end{proof}

\subsection*{Acknowledgements}
We are grateful to Alessandro Paolini for providing us with useful data which allowed us to detect the pattern explained by the main result of this note. We would also like to thank Matthew Fayers, Gabriel Navarro, Mark Wildon and Alex Zalesski for helpful comments on a previous version of this paper.


\end{document}